\documentclass[a4paper,10pt]{amsart}
\usepackage{bbm,amssymb,amsmath, pinlabel}

  \usepackage{hyperref}
 
\input xy
\xyoption{all}

\usepackage{color}

\newtheorem{theorem}{Theorem}[section]
\newtheorem{lemma}[theorem]{Lemma}
\newtheorem{proposition}[theorem]{Proposition}
\newtheorem{corollary}[theorem]{Corollary}

\theoremstyle{definition}
\newtheorem{definition}[theorem]{Definition}
\newtheorem{remark}[theorem]{Remark}

\numberwithin{equation}{section}

\newcommand{\gras}[1]{{\mathbb #1}} 

\newcommand{\Z}{\gras{Z}}

\newcommand{\C}{\gras{C}}  

\newcommand{\bP}{\gras{P}}

\def\elem(#1,#2){  \{ \frac{#1}  {\overline {\ #2\ }} \} }

\title{On the smoothings of non-normal isolated surface singularities}
\author{Patrick Popescu-Pampu}
   \address{Universit{\'e} Lille 1, UFR de Maths., B\^atiment M2\\
     Cit\'e Scientifique, 59655, Villeneuve d'Ascq Cedex, France.}
   \email{patrick.popescu@math.univ-lille1.fr}

\begin{document}

\begin{abstract}
    We show that isolated surface singularities which are non-normal may have Milnor fibers 
    which are non-diffeomorphic to those of their normalizations. Therefore, 
    non-normal isolated singularities enrich the collection of Stein fillings 
    of links of normal isolated singularities. We conclude with 
    a list of open questions related to this theme. 
\end{abstract}

This paper was published in \emph{Proceedings of Singularities in Geometry and 
Applications III}, Edinburgh, Scotland, 2013.  J.-P. Brasselet, P. Giblin, V. Goryunov eds. 
Journal of Singularities {\bf 12} (2015), 164-179. http://www.journalofsing.org/volume12/article12.html
\bigskip

\maketitle

\tableofcontents

\section{Introduction}  
\label{sec:intro}

Let $(S,0)$ be a germ of irreducible complex analytic space with isolated 
singularity. Varchenko \cite{V 80} proved that there is a well-defined 
isomorphism class of contact structures on its link (or \emph{boundary}, as we prefer to 
call it in this paper). Following the terminology introduced in \cite{CNP 06}, we 
say that a contact manifold which appears in this way is \emph{Milnor fillable}. 
We use the same name if we forget the contact structure: namely, an oriented odd-dimensional 
manifold is Milnor-fillable if and only if it is orientation-preserving diffeomorphic 
to the boundary of an isolated singularity. 

If $(S,0)$ is \emph{smoothable}, that is, if there exist deformations of it with smooth generic 
fibers, then there exist representatives of such fibers -- the so-called \emph{Milnor fibers} of the 
deformation -- which are 
Stein fillings of the contact boundary of the singularity. Milnor fibers associated to arbitrary 
smoothings were mainly studied 
till now for \emph{normal} surface singularities. When they are rational homology balls, 
they are used for the operation of \emph{rational blow-down} introduced by 
Fintushel and Stern \cite{FS 97} and generalized by Stipsicz, Szab\'o, Wahl \cite{SSW 08}. 
Due to the efforts of several researchers, the \emph{normal} surface singularities which 
have smoothings whose Milnor fibers are rational homology balls are now completely classified. 
See \cite{PSS 13} and \cite{BS 13} for details on this direction of research. 

In another direction, there are results which classify \emph{all} the possible Stein fillings 
(independently of their homology) up to 
diffeomorphisms,  for special kinds of singularities: Ohta and Ono did this for 
simple elliptic singularities \cite{OO 03} and simple singularities \cite{OO 05}, 
Lisca for cyclic quotient singularities \cite{L 08}, Bhupal and Ono \cite{BO 12} for the remaining  
quotient surface singularities. 

If $(S,0)$ is fixed, the existence of a holomorphic versal deformation, proved by Grauert 
\cite{G 72}, shows that, up to diffeomorphisms, there is only a finite number of Stein fillings of its contact boundary which appear as Milnor fibers of its smoothings. For all the previous 
classes of singularities, there is also a finite number of Stein fillings and even of 
strong symplectic fillings. This fact is not general. Ohta and Ono \cite{OO 08} 
showed that there exist Milnor fillable contact $3$-manifolds which admit an infinite 
number of minimal strong symplectic fillings, pairwise not homotopy equivalent. Later, 
Akhmedov and Ozbagci 
\cite{AO 12} proved that  there exist Milnor fillable contact $3$-manifolds 
which admit even an infinite number of Stein fillings pairwise non-diffeomorphic, 
but homeomorphic. Moreover, by varying the 
contact $3$-manifold, the fundamental groups of such fillings 
exhaust all finitely presented groups. For details on this direction of research, 
one may consult Ozbagci's survey \cite{O 14}. 

For simple singularities (see \cite{OO 05}) and for cyclic quotients (see \cite{NP 10}),  
all Stein fillings are diffeomorphic to the Milnor fibers 
of the smoothings of a singularity with the given contact link (in each case there is only one such 
singularity, up to isomorphisms). 
By contrast, for \emph{simple elliptic singularities}, there exist Stein fillings of their contact boundary 
which are not diffeomorphic to a Milnor fiber, but to the total space of their minimal resolution. 

For instance, in the case of those simple elliptic singularities which are 
not smoothable (which means, by a theorem of Pinkham \cite{P 74}, that the exceptional 
divisor of the minimal resolution is an elliptic curve with self-intersection $\leq -10$),  
there is only one Stein filling, which is diffeomorphic to the total space of the  
minimal resolution. 

We explain here (see Section \ref{sec:sell}), 
that \emph{this total space is diffeomorphic to the Milnor fiber of a smoothing 
of a \emph{non-normal} isolated surface singularity, whose normalization is the given 
non-smoothable simple elliptic singularity}. We do this by using the simplest technique of 
construction of smoothings, which was called ``\emph{sweeping out the cone by hyperplane 
sections}'' by Pinkham \cite{P 74}. 
This has the advantage of showing that those Milnor fibers are in fact 
diffeomorphic to affine algebraic surfaces. 

More generally, the results of Laufer \cite{L 81} 
and Bogomolov and de Oliveira \cite{BO 97} show that, for any \emph{normal} 
surface singularity $(S,0)$, there is a smoothing of an isolated surface singularity 
whose Milnor fiber is diffeomorphic to the minimal resolution of $(S,0)$ 
(see Proposition \ref{milnres}). 

We wrote this paper in order to emphasize \emph{the problem of the topological 
study of the smoothings of non-normal isolated singularities}. 
Let us mention that Jan Stevens has a manuscript \cite{S 09} which emphasizes  the 
\emph{algebraic} aspects of the deformation theory of such singularities. 

We have in mind as potential readers graduate students specializing either in singularity 
theory or in contact/symplectic topology, therefore we explain several notions and 
facts which are well-known to specialists of either field, but maybe not to both. 

Let us describe briefly the contents of the various sections. 
In Section \ref{sec:genorm} we explain basic facts about normal surface singularities, their resolutions and the classes 
of rational, minimally elliptic and simple elliptic singularities. In Section 
\ref{sec:gen}  we explain the basic 
notions about deformations needed in the sequel. 
In Section \ref{sec:sweepcone} we explain the technique of sweeping out a cone by 
hyperplane sections and the reason why one does not necessarily get 
in this way a normal singularity, even if the starting singularity is normal. 
In Section \ref{sec:sell} we continue with material about 
very ample curves on ruled surfaces, and we apply it to the construction of 
the desired smoothings.  In the last section, we list a series of open 
questions which 
we consider to be basic for the knowledge of the topology of deformations of 
\emph{isolated non-normal singularities}.

\medskip
{\bf Acknowledgements.} I benefited from conversations with Burak Ozbagci, 
   Jan Stevens and Jean-Yves Welschinger. I am grateful also to the referee for his 
   remarks and to J\'anos Koll\'ar who, after seeing the first version of this paper put 
   on ArXiv, communicated me a list of papers dealing 
   at least partially with smoothings of non-normal isolated surface singularities  
   (see Remark \ref{Kollar}). This research was partially supported 
   by the grant ANR-12-JS01-0002-01 SUSI and by Labex CEMPI (ANR-11-LABX-0007-01).

\vfill

\pagebreak

\section{Generalities on normal surface singularities}
\label{sec:genorm}

In this section we recall the basic properties and classes of normal 
surface singularities which are needed in the sequel. More detailed 
introductions to the study of normal surface singularities are contained 
in \cite{R 97}, \cite{N 99}, \cite{N 04}, \cite{W 00}, \cite{P 07}. 
\medskip

Recall first the basic definition, valid in arbitrary dimension:

\begin{definition} \label{norm}
  Let $(X,x)$ be a germ of  reduced  complex analytic space. 
  It is called {\bf normal} if and only if its local ring of holomorphic functions 
  is integrally closed in its total ring of fractions. 
\end{definition}

Normality may be characterized also in the following ways (see \cite[Page 81]{GPR 94}):

\begin{proposition}
   Let $(X,x)$ be a germ of  reduced  complex analytic space.  The following 
   statements are equivalent:
     \begin{enumerate}
         \item $(X,x)$ is normal.
         \item The singular locus $S(X)$ of $X$ is of codimension at least $2$ and any 
           holomorphic function on $X \setminus S(X)$ extends to a holomorphic 
           function on $X$.
         \item Every bounded holomorphic function on $X \setminus S(X)$ extends 
             to a holomorphic  function on $X$.
     \end{enumerate}
\end{proposition}

Using this proposition, it is easy to show that:

\begin{corollary}  \label{normext}
   If the reduced germ $(X,x)$ is normal, then  
   every continuous function  
        $$f : (X,x) \to (Y,y),$$  
    where $(Y,y)$ 
  is another holomorphic germ, is necessarily holomorphic whenever it is holomorphic 
  on the complement of a nowhere dense closed analytic subspace $(X',x) \subset (X,x)$. 
\end{corollary}

Any reduced germ has a canonical \emph{normalization}, 
whose multilocal ring (direct sum of a finite collection of local rings) 
is the integral closure of the initial local ring in its 
total ring of fractions. It may be characterized in the following way:

\begin{proposition}  \label{charnorm}
    Let $(X,x)$ be a germ of  reduced complex analytic space. 
   There exists, up to unique isomorphism above $(X,x)$, a unique 
   finite morphism $\nu : (\tilde{X}, \tilde{x}) \to (X, x)$ from a finite disjoint 
   union of germs to $(X,x)$ (here $\tilde{x}$ denotes a finite set of points), such that:
     \begin{itemize}
        \item $\nu$ is an isomorphism outside the non-normal locus of $X$. 
        \item $\tilde{X}$ is normal. 
     \end{itemize}
\end{proposition}

Therefore, normal germs are necessarily irreducible. The normalization separates 
the irreducible components and eliminates the components of their singular 
loci which are of codimension $1$. In particular, 
normal curve singularities are precisely the smooth ones and normal surface singularities are necessarily isolated. The converse is not true in any dimension 
(see the explanations given in the proof of Proposition \ref{3fold}). 
Nevertheless, complete intersection isolated singularities of dimension $2$ or higher 
are necessarily normal (being Cohen-Macaulay, see the same proof). 
This is the reason why it is more difficult to exhibit examples of isolated 
non-normal singularities in dimension $2$ or higher than in dimension $1$. 

Normal singularities are of fundamental importance even if one 
is interested in non-normal ones: a way to study them is through their morphism $\nu$ 
of {\em normalization}, characterized in the previous proposition. For much more details 
about normal varieties and the normalization maps, one may consult Greco's book 
\cite{G 78}.

\medskip

One has a preferred family of representatives of any germ with isolated singularity:

\begin{definition} \label{miln}
    Let $(X,x)$ be a germ of  reduced and irreducible complex analytic space with 
    isolated singularity. Choose a representative of it embedded in 
     $(\C^n, 0)$. Consider the euclidean sphere $\mathbb{S}^{2n -1}(r) \subset 
     \C^n$ of radius $r >0$, centered at $0$. Denote by $\mathbb{B}^{2n}(r)$ 
    the ball bounded by it.  A ball  $\mathbb{B}^{2n}(r_0)$ is called a 
    {\bf Milnor ball} if all the spheres of radius $r \in (0, r_0]$ are transversal to the representative. 
    In this case, the intersection $X \cap \mathbb{B}^{2n}(r_0)$ is called a 
    {\bf Milnor representative} of the germ and  $X \cap \mathbb{S}^{2n-1}(r_0)$ 
    is the {\bf boundary} of the germ. 
\end{definition}

The boundary is independent, up to diffeomorphisms preserving the orientation, 
of the choices done in this construction (see Looijenga \cite{L 84}). We will denote 
its oriented diffeomorphism type, or a representative of it, by 
$\partial (X,x)$. One may show, moreover, that the boundary of an isolated singularity 
is isomorphic to the boundary of its normalization. This may seem obvious intuitively, 
as the normalization morphism is in this case an isomorphism outside the singular point, 
but one has to work more, because the lift to the normalization of the euclidean distance 
function serving to define the intersections with spheres for the initial germ are not 
euclidean distance functions for the normalization. For a detailed treatment of this issue, 
see \cite{CNP 06}. 

Let us fix a Milnor ball $\mathbb{B}^{2n}(r_0)$.  At each point of the representative 
$X \cap \mathbb{S}^{2n-1}(r_0)$ of $\partial (X,x)$, consider the maximal subspace 
of the tangent space which is invariant by the complex multiplication. It is a (real) 
hyperplane, canonically oriented by the complex multiplication. This field of 
hyperplanes is moreover a \emph{contact structure}, as a consequence of the fact that 
the spheres by which we intersect are strongly pseudoconvex. 
In fact, this oriented contact manifold is also independent of the choices. 
We call it the {\bf contact boundary} $(\partial (X,x), \xi(X,x))$ of the singularity 
$(X,x)$ (for details, see \cite{CNP 06}).  In the same reference, we introduced the following 
terminology:

\begin{definition} \label{milnfill}
  An oriented (contact) manifold is called {\bf Milnor fillable} if it is isomorphic 
  to the (contact) boundary of an isolated singularity. 
\end{definition}

\medskip
From now on, we will restrict to surfaces. One of the most important tools to study 
them is: 

\begin{definition} \label{resol}
      Let $(S,0)$ be a normal surface singularity which is not smooth. 
      A {\bf resolution} of it is 
      a morphism $\pi : (\Sigma, E) \to (S, 0)$, where $E$ denotes the preimage 
      of $0$ by $\pi$, such that:
        \begin{itemize}
            \item  $\pi$ is proper;
            
            \item $\Sigma$ is smooth;
            
            \item $\pi$ is an isomorphism from $\Sigma \setminus E$ to $S \setminus 0$. 
        \end{itemize}
        The subset $E$ of $\Sigma$, which is always a connected divisor, is called 
        the {\bf exceptional divisor} of $\Sigma$. 
        If $E$ is a divisor with normal crossings whose irreducible components are smooth, 
        we say that $\pi$ is a {\bf simple normal crossings (snc)} resolution. In this last case, 
        the {\bf dual graph} of the resolution has as vertices the irreducible components 
        of $E$, the edges being in bijection with the intersection points of those components. 
\end{definition}

Note that the hypothesis of having simple normal crossings prohibits the existence of loops 
in the dual graphs, but not that of multiple edges. In fact, the number of edges between two 
vertices is equal to the intersection number of the corresponding components. 

There always exist resolutions. Moreover, there is always a \emph{minimal} snc 
resolution, unique up to unique isomorphism above $(S, 0)$, the minimality meaning 
that any other snc resolution factors through it. It is this resolution which is most 
widely used for the topological study of the boundary of the singularity. Nevertheless, 
for its algebraic study, sometimes it is important to work with the \emph{minimal 
resolution}, in which we don't ask any more the exceptional resolution to have normal 
crossings or smooth components (see an example in Theorem \ref{minel}). 
It is again a theorem that such a resolution also exists up to unique isomorphism.

If $\pi$ is a resolution of $(S,0)$, denote by $\mathrm{Eff}(\pi)$ the free abelian 
semigroup generated by the irreducible components of its exceptional divisor, that is, 
the additive semigroup of the integral effective divisors supported by $E$. 
If $Z_1, Z_2 \in \mathrm{Eff}(\pi)$, we say that $Z_1$ {\bf is less than} 
$Z_2$ if $Z_2 - Z_1$ is also effective and $Z_1 \neq Z_2$. We write then $Z_1 < Z_2$.

\begin{proposition}  \label{numcycl}
   Let $\pi$ be any resolution of the normal surface singularity $(S,0)$. 
  There exists a non-zero cycle $Z_{num} \in \mathrm{Eff}(\pi)$, 
  called the {\bf numerical cycle} of $\pi$, 
  which intersects non-positively all the irreducible components  of $E$, and which is 
  less than all the other cycles having this property. 
\end{proposition} 

By definition, the numerical cycle is unique, once the resolution is fixed. It was defined 
first by M. Artin \cite{A 66}, and Laufer \cite{L 72} gave an algorithm to compute it. 

We will need a second cycle supported by $E$, this time with \emph{rational} coefficients, 
possibly non-integral. 

\begin{proposition}
     Let $\pi: (\Sigma, E)\to (S,0)$ be any resolution of the normal surface singularity $(S,0)$. 
    There exists a unique cycle $Z_K$ supported by  $E$, with rational coefficients,  
    such that $Z_K \cdot E_i = - K_{\Sigma}\cdot E_i$ for any component $E_i$ 
    of $E$. It is called the {\bf anticanonical cycle} of $\pi$. Here 
    $K_{\Sigma}$ denotes any canonical divisor of $\Sigma$. 
\end{proposition}

The canonical divisors on $\Sigma$ are the divisors of the meromorphic $2$-forms 
on a neighborhood of $E$ in $\Sigma$. Such forms are precisely the lifts of the meromorphic 
$2$-forms on a neighborhood of $0$ in $S$. Of special importance are the normal 
surface singularities admitting such a $2$-form which, moreover, is holomorphic and 
does not vanish on $S \setminus 0$: 

\begin{definition}
   An isolated surface singularity $(S,0)$ is {\bf Gorenstein} if it is normal and if 
   it admits a non-vanishing holomorphic form of degree $2$ on $S \setminus 0$. 
\end{definition}

In fact, isolated complete intersection surface singularities are not only normal, but also 
Gorenstein. We remark that the topological types of Gorenstein 
isolated surface singularities are known by \cite{P 11}, but it is an open question to 
describe the topological types of those which are complete intersections or hypersurfaces. 

Both the anticanonical cycle and the notion of Gorenstein singularity are defined using 
differential forms of degree $2$. Such forms are also useful to define several 
important notions of \emph{genus}:

\begin{definition} \label{geomgen}
   Let $(S,0)$ be a normal surface singularity. Its {\bf geometric genus} 
   $p_g(S,0)$ is equal to the dimension of the space of holomorphic $2$-forms 
   on $S \setminus 0$, modulo the subspace of forms which extend 
   holomorphically to a resolution of $S$. 
   
   If $Z$ is a compact divisor on a smooth complex surface $\Sigma$, its 
   {\bf arithmetic genus} $p_a(Z)$ is equal to $1 + \dfrac{1}{2}Z\cdot (Z + K_{\Sigma})$. 
\end{definition}

In the same way as the \emph{rational} curves are those of smooth 
algebraic curves of genus (in the usual 
Riemannian sense) $0$, M. Artin  \cite{A 66} defined:

\begin{definition}
   A normal surface singularity is {\bf rational} if its geometric genus is $0$. 
\end{definition}

By contrast with the case of curves, there is an infinite set of topological types 
of rational surface singularities. A basic property of them is that their minimal 
resolutions are snc, that all the irreducible components of their exceptional divisors 
are rational curves, and that their dual graphs are trees. But this is not 
enough to characterize them. In fact, as proved by M. Artin \cite{A 66}:

\begin{proposition} \label{charat}
   Let $(S,0)$ be a normal surface singularity and let 
   $\pi: (\Sigma, E)\to (S,0)$ be any resolution of  it. Then $(S,0)$ is rational 
   if and only if $p_a(Z_{num}) =0$. 
\end{proposition}

The reader interested in the combinatorics of rational surface singularities 
may consult L\^e and Tosun's paper  \cite{LT 04} and Stevens' paper \cite{S 13}. 

The singularities on which we focus in the sequel are not rational, as 
their resolutions contain non-rational exceptional curves:

\begin{definition} \label{simplell}
   A normal surface singularity is called {\bf simple elliptic} if the exceptional 
   divisor of its minimal resolution is an elliptic curve.
\end{definition}

Simple elliptic singularities are necessarily Gorenstein, as a consequence of the following 
theorem of Laufer \cite[Theorems 3.4 and 3.10]{L 77}: 

\begin{theorem} \label{minel} 
  Let $(S,0)$ be a normal surface singularity. 
  Working with its minimal resolution, the following facts are equivalent:
  \begin{enumerate}
     \item One has $p_a(Z_{num})=1$ and $p_a(D)<1$ for all $0<D<Z_{num}$.
     
     \item The fundamental and anticanonical cycles are equal: $Z_{num}= Z_K$. 
     
     \item \label{motiv} One has $p_a(Z_{num})=1$ and any connected
       proper subdivisor of $E$  
       contracts to a rational singularity.
       
     \item $p_g(S,0)=1$ and $(S,0)$ is Gorenstein.       
  \end{enumerate}
\end{theorem}

Laufer introduced a special name (making reference to condition
(\ref{motiv})) for the singularities satisfying one of
the previous  conditions:

\begin{definition} \label{defminel}
  A normal surface singularity satisfying one of the  equivalent 
  conditions stated in Theorem \ref{minel} is called a {\bf minimally 
  elliptic} singularity. 
\end{definition}

In fact, as may be rather easily proved using characterization (\ref{motiv}) of 
minimally elliptic singularities, the simple elliptic singularities are precisely 
the minimally elliptic ones which admit resolutions whose exceptional 
divisors have at least one non-rational component.

\bigskip

\section{Generalities on deformations and smoothings of isolated singularities}
\label{sec:gen}

In this section 
we recall the basic definitions and properties about deformations 
of isolated singularities which are needed in the sequel. 
For more details, one may consult Looijenga \cite{L 84}, 
Looijenga \& Wahl \cite{LW 86}, Stevens \cite{S 03}, 
Greuel, Lossen \& Shustin \cite{GLS  07} and N\'emethi \cite{N 13}. 
\medskip

\begin{definition} \label{deform} \index{Deformations}
  Let $(X,x)$ be a germ of a complex analytic space. A 
  {\bf deformation} of $(X,x)$ is a germ of {\em flat} morphism 
  $\psi:(Y,y)\rightarrow (S,s)$ together with an 
  isomorphism between $(X,x)$ and the special fiber $\psi^{-1}(s)$. 
  The germ $(S,s)$ is called the \textbf{base} of the deformation. 
\end{definition}

For example, when $X$ is reduced, $f\in m_{X,x}$ is flat as a morphism
$(X,x) \stackrel{f}{\rightarrow} (\C,0)$ if and only if $f$ does not
divide zero, that is, if and only if $f$ does not vanish on a whole irreducible component 
of $(X,x)$. Such deformations over germs of smooth curves are called
$1$-{\em parameter deformations}. The simplest example is obtained when
$X=\C^n$. Then one gets the prototypical situation considered by
Milnor \cite{M 68}.  

In general, to think about a flat morphism as a ``deformation'' means to
see it as a family of continuously varying fibers (in the sense that 
their dimension is locally constant, without blowing-up phenomena) 
and to concentrate on a particular fiber, the nearby ones being seen as 
``deformations'' of it. From such a family, 
one gets new families by rearranging the
fibers, that is, by {\em base change}. One is particularly interested
in the situations where there exist families which generate all other
families by such base changes. 
The following definition is a reformulation of \cite[Definition 1.8,
page 234]{GLS 07}:

\begin{definition} \label{miniv}
  \begin{enumerate}
      \item A deformation of $(X,x)$ is {\bf
          complete}\index{Complete deformation} if any other
        deformation is  
  obtainable from it  by a base-change.  

      \item A complete deformation $\psi$ of $(  X,x)$ is
  called {\bf versal}\index{Versal deformation} if for any other
  deformation over a base $(T,t)$ 
  and identification  
  of the induced deformation over a subgerm $(T',t)
  \hookrightarrow (T,t)$ with a pull-back from $\psi$, one may extend
  this identification with a pull-back from $\psi$ over all $(T,t)$. 

       \item A versal deformation is {\bf
    miniversal}\index{Miniversal deformation} if  
  the Zariski tangent space of its base $(S,s)$ has the smallest possible 
  dimension. 
  \end{enumerate}
\end{definition} 

When the miniversal deformation exists, its base space is unique {\em
  up to non-unique isomorphism} (only the tangent map to the 
isomorphism is unique). For this reason, one does not speak about a
{\em universal} deformation, and was coined the word ``miniversal'',
with the variant ``semi-universal''.  

In many references, versal deformations are defined as the complete
ones in the previous definition. Then is stated the theorem
that the base of a versal deformation is isomorphic to the product of
the base of a miniversal deformation and a {\em smooth} germ. But with
this weaker definition the result is
false. Indeed, starting from a complete deformation, by doing the
product of its base with {\em any} germ (not necessarily
smooth) and by taking the pull-back, we would get again a complete
deformation.  This shows  that a complete deformation is not necessarily
versal. Nevertheless, the theorem stated before is true if one uses the
previous definition of \emph{versality}. 

Not all germs admit versal deformations. But those with isolated singularity 
do admit, as was proved by  Schlessinger \cite{S 68} for \medskip{formal} deformations 
(that is, over spectra of formal analytic algebras), then by Grauert \cite{G 72} for holomorphic 
ones (an important point of this theorem being that one has to work with general 
analytic spaces, possibly non-reduced):

\begin{theorem} \label{exvers}
  Let $(X,x)$ be an isolated singularity. Then the miniversal 
  deformation  
  exists and is unique up   to (non-unique) isomorphism. 
\end{theorem}

One may extend the notion of deformation by allowing bases of infinite dimension. Then 
even the germs with non-isolated singularity have versal deformations (see Hauser's 
papers \cite{H 83}, \cite{H 85}). 

In the sequel we will be interested in deformations with smooth generic fibers:

\begin{definition} \label{smooth}
A {\bf smoothing} of an isolated singularity $(X,x)$ is a $1$-parameter
deformation whose generic fibers are smooth. 
A {\bf smoothing component} of $(X,x)$ is an
irreducible component  
  of the reduced miniversal 
  base space  over which the generic  fibers are smooth. 
\end{definition}

Isolated complete intersection singularities have a miniversal deformation
$(Y,y) \stackrel{\psi}{\rightarrow} (S,s)$ such that both $Y$ and $S$
are smooth, therefore irreducible (see \cite{L 84}). 
In general, the {\em reduced miniversal base} $(S_{red},s)$ may be
reducible. The first example of this 
phenomenon was discovered by Pinkham 
\cite[Chapter 8]{P 74}:

\begin{proposition} \label{firstex} 
  The germ at the origin of the cone over the rational normal curve of
  degree $4$ in $\bP^4$ has a reduced miniversal base space with two
  components, both being smoothing ones. 
\end{proposition}

Not all isolated singularities are smoothable. The most extreme case
is attained with {\em rigid}\index{Rigid singularities} singularities,
which are not deformable 
at all in a non-trivial way. For example, quotient singularities of
dimension $\geq 3$ are rigid (Schlessinger \cite{S  71}). 

In \cite{P 08} we proved a purely
topological obstruction to smoothability for singularities of
dimension $\geq 3$. In dimension $2$ no such criterion is known 
for {\em all} normal singularities. But there exist
such obstructions for {\em Gorenstein} normal surface singularities 
as a consequence of the following theorem of Steenbrink \cite{S 83}:

\begin{theorem} \label{obsgor} 
    Let $(X,x)$ be a Gorenstein normal surface singularity. If it is
    smoothable, then:
  \begin{equation} \label{mumin}
      \mu_- = 10 \  p_g(X,x) - b_1(\partial(X,x)) + (Z_K^2 + |I|).
  \end{equation}
\end{theorem}

In the preceding formula, $\mu_-$ denotes the negative part of the index of
    the intersection form on the second homology group of \emph{any} Milnor
    fiber (see Theorem \ref{topdeform} below) 
    and $b_1(\partial(X,x))$ denotes the first Betti
number of the boundary of $(X,x)$. It may be computed from any
snc resolution with exceptional divisor $E=\sum_{i \in I}
E_i$ as:
 $$b_1(\partial(X,x))= b_1(\Gamma) + 2 \sum_{i \in I} p_i,$$
 where $p_i$ denotes the genus of $E_i$ and $\Gamma$ denotes the 
 dual graph of $E$. 
The term $Z_K^2 + |I|$ may also be computed using any snc 
resolution, and is again a topological invariant of the singularity. 

The previous theorem implies that the expression in the right-hand
side of (\ref{mumin}) is $\geq 0$, which gives non-trivial
obstructions on the topology of smoothable normal Gorenstein
singularities. For example, it shows that:

\begin{proposition} \label{sellsmooth} 
  Among simple elliptic singularities,
   the smoothable ones have minimal resolutions whose 
    exceptional divisor is an elliptic curve with self-intersection $\in
   \{-9,-8,...,-1\}$. 
\end{proposition}

Proposition \ref{sellsmooth} 
has been proved first in another way by Pinkham \cite[Chapter 7]{P 74}.

\medskip
 
Let us look now at the topology of the generic fibers above a
smoothing component. We want to localize the study of the family in
the same way as Milnor localized the study of a function on $\C^n$
near a singular point. This is possible (see Looijenga \cite{L 84}):

\begin{theorem} \label{topdeform}
Let $(X,x)$ be an isolated singularity. Let $(Y,y)
\stackrel{\psi}{\rightarrow} (S,s)$ be a miniversal deformation of
it.  There exist (Milnor) representatives $Y_{red}$ and $S_{red}$ 
of the reduced total and base spaces of $\psi$ such that the 
restriction $\psi:\partial Y_{red} \cap\psi^{-1}(S_{red}) 
\rightarrow S_{red}$ is a trivial $C^{\infty}$-fibration. 
Moreover, one may choose those representatives such that over each
smoothing component $S_i$, one gets a locally trivial $C^{\infty}$-fibration 
$\psi: Y_{red} \cap\psi^{-1}(S_i) \rightarrow S_i$
outside a proper analytic subset.
 \end{theorem}
 
Hence, for each smoothing component $S_i$,  the oriented diffeomorphism type  
of the oriented manifold with boundary  
 $(\pi^{-1}(s)\cap Y_{red},\pi^{-1}(s)\cap \partial Y_{red})$ does not depend 
 on the choice of the generic element   $s\in S_i$: it
 is called the {\bf Milnor fiber} of that
 component. Moreover,  its boundary   is canonically 
identified with the boundary of $(X,x)$ \emph{up to isotopy}.  
In particular, the Milnor fiber of a  
smoothing component is diffeomorphic to a Stein filling of the contact
boundary ($\partial(X,x), \xi(X,x))$.   

Greuel and Steenbrink \cite{GS 83} proved 
 the following topological restriction on the Milnor fibers 
of \emph{normal} isolated singularities (of any dimension):

\begin{theorem} \label{vanish}
   Let $(X,x)$ be a normal isolated singularity. Then all its Milnor fibers 
   have vanishing first Betti number. 
\end{theorem}

This is not true for non-normal isolated surface singularities, as may be seen 
for instance from the examples we give in the last section (see Remark 
\ref{nonorm}).

For singularities which are not complete intersections, 
it is in general difficult even to construct non-trivial 
deformations or to decide if there exist smoothings. There is nevertheless 
a general technique of construction of smoothings, applicable to 
\emph{germs of affine cones at their vertices}. Next section is dedicated to it.

\bigskip

\section{Sweeping out the cone with hyperplane sections} 
\label{sec:sweepcone}

In this section we recall Pinkham's method of construction of smoothings 
by ``\emph{sweeping out the cone with hyperplane 
sections}''. It may be applied to the germs of affine cones at their vertices. 
The reader may follow the explanations on Figure \ref{fig:Sweep}.

\begin{figure}[h!] 
\vspace*{6mm}
\labellist \small\hair 2pt 
\pinlabel{$A$} at 124 383
\pinlabel{$B$} at 165 360
\pinlabel{$H$} at 250 403
\pinlabel{$O$} at 182 150
\pinlabel{$C_A$} at 268 262
\pinlabel{$C_B$} at 257 160
\pinlabel{$C_H$} at 160 27
\pinlabel{$\bP(V)$} at 32 300
\pinlabel{$V$} at 32 50
\pinlabel{$W$} at 350 200

\endlabellist 
\centering 
\includegraphics[scale=0.50]{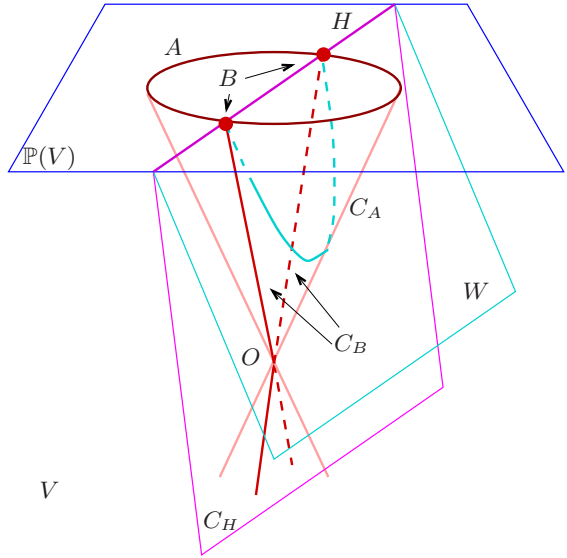} 
\caption{Sweeping the cone with hyperplane sections} 
\label{fig:Sweep}
\end{figure}

\medskip

Let $V$ be a complex vector space, whose projectivisation is denoted $\bP(V)$: 
set-theoretically, it consists of the lines of $V$. 
More generally, 
we define the \emph{projectivisation} $\bP(\mathcal{V})$ of a 
vector bundle $\mathcal{V}$ as the set of lines contained in the various 
fibers of the bundle. This notion will be used in the next section (see Remark 
\ref{dualproj}).

Let $A$ be a smooth subvariety of $\bP(V)$. Denote by $C_A 
\hookrightarrow V$ the \emph{affine cone} over it, and by 
$\overline{C_A} \hookrightarrow \overline{V}$ the associated \emph{projective cone}. 
Here $\overline{V}$ denotes the projective space of the same dimension as $V$, 
obtained by adjoining $\bP(V)$ to $V$ as hyperplane at infinity. That is:
   $$\overline{V} = \bP(V \oplus \C) = V \cup \bP(V).$$
The projective cone $\overline{C_A} = C_A \cup A$ is the Zariski closure of $C_A$ in 
 $\overline{V}$. The \emph{vertex} of either cone is the origin $O$ of $V$. 
 
 Assume now that $H\hookrightarrow \bP(V)$ is a projective hyperplane which intersects 
$ A$ \emph{transversally}. Denote by:
  $$B := H \cap A$$
 the corresponding hyperplane section of $A$. The affine cone $C_H$ over $H$ is 
 the linear hyperplane of $V$ whose projectivisation is $H$. 
 The associated projective cone $\overline{C_H} \hookrightarrow \overline{V}$ 
 is a projective hyperplane of $\overline{V}$. 
 
 Let $L$ be the pencil of hyperplanes of $\overline{V}$ generated by $\bP(V)$ 
 and $\overline{C_H}$. That is, it is the pencil of hyperplanes of $\overline{V}$ 
 passing through the ``axis'' $H$. In restriction to $V$, it consists 
 in the levels of any linear 
 form $f : V \to \C$ whose kernel is $C_H$. The $0$-locus of $f|_{C_A}$ is 
 the affine cone $C_B$ over $B$. 

As an immediate consequence of the fact that $H$ intersects $A$ 
transversally, we see that $C_B$ has an isolated singularity at $0$ 
and that all the non-zero levels of $f|_{C_A}$ are smooth. This shows that:

\begin{lemma}
  The map $f|_{C_A} : C_A \to \C$ gives  a smoothing of the isolated singularity 
  $(C_B, O)$. 
\end{lemma}

Such are the smoothings obtained by ``\emph{sweeping out the cone with hyperplane 
sections}'', in the words of Pinkham \cite[Page 46]{P 74}. 
It is probably the easiest way to construct smoothings, which 
explains why a drawing similar to the one we include here 
was represented on the cover of Stevens' book \cite{S 03}. 

Since the complement $C_A \setminus O$ of the vertex in the cone $C_A$ is homogeneous 
under the natural $\C^*$-action by scalar multiplication on $V$, the Milnor fibers 
of $f|_{C_A} : (C_A, O) \to (\C, 0)$ are diffeomorphic to the global (affine) fibers 
of  $f|_{C_A} : C_A \to \C$. Those fibers are the complements $(W  \cap \overline{C_A}) 
\setminus B$, for the members $W$ of the pencil $L$ different from $\overline{C_H}$ 
and $\bP(V)$. But the only member of this pencil which intersects 
$\overline{C_A}$ non-transversally is $\overline{C_H}$, which shows that 
the pair $(W  \cap \overline{C_A}, B)$ is diffeomorphic to 
$(\bP(V) \cap \overline{C_A}, B)= (A, B)$. Therefore:

\begin{proposition} \label{complsect}
    The Milnor fibers of the smoothing $f|_{C_A} : (C_A, O) \to (\C, 0)$ of the 
    singularity $(C_B, O)$ are diffeomorphic to 
    the affine subvariety $A \setminus B$ of the affine space $\bP(V) \setminus H$. 
\end{proposition}

The previous method may be applied to construct smoothings of 
germs of affine cones $C_B$ at their vertices. 
In order to apply it, one has therefore 
to find another subvariety $A$ of the same projective space, containing $B$, and such that 
$B$ is a section of $A$ by a hyperplane intersecting it transversally. In general, this is 
a difficult problem. 

The important point to be understood here 
is that, \emph{even if $(C_A, O)$ is normal, this is not necessarily the case for its 
hyperplane section $(C_B, O)$}. More generally, if $(Y,y)$ is a \emph{normal} 
isolated singularity and $f:(Y, y) \to (\C, 0)$ is a holomorphic function such that 
the germ $(f^{-1}(0), y)$ is reduced and with isolated singularity, it is not necessarily normal. 
In dimension $3$, in which we are especially interested in here, 
something special happens:

\begin{proposition} \label{3fold}
   Assume that $(Y,y)$ is a normal germ of $3$-fold, with isolated singularity, 
   and that $(f^{-1}(0), y)$ has also an isolated singularity. Then 
   $(f^{-1}(0), y)$ is normal if and only if $(Y,y)$ is Cohen-Macaulay. 
\end{proposition}

\begin{proof}
       Let us explain first basic intuitions about \emph{Cohen-Macaulay germs}. 
      This notion appears naturally if one 
      studies singularities using successive hyperplane sections. 
      Intrinsically speaking, a hyperplane section of a germ $(Y,y)$ is defined as the zero-locus 
      of a function $f \in \mathfrak{m}$, where $\mathfrak{m}$ is the maximal 
      ideal of the local ring $\mathcal{O}$ of the germ, endowed with the analytic 
      structure given by the quotient local ring $\mathcal{O} / (f)$. This section is of dimension 
      at least $\dim(Y,y) -1$. Dimension drops necessarily if $f$ \emph{is not a 
      divisor of $0$} in $\mathcal{O}$. Do such functions exist? Not necessarily. 
      But if they exist, we take the hyperplane section and we repeat the process. 
      $(Y,y)$ is called \emph{Cohen-Macaulay} if it is possible to drop in this way 
      iteratively the dimension till arriving at an analytical space of dimension $0$ 
      (that is, set-theoretically, at the point $y$). 
      
      For the basic properties of the previous notion, 
      one may consult \cite{GPR 94} or \cite{G 78}.  Here we will need only the following facts: 
        \begin{enumerate}
           \item  \label{CMsect}
      If a germ is Cohen-Macaulay, then for any 
      $f \in \mathfrak{m}$ non-dividing $0$, the associated 
      hyperplane section $(f^{-1}(0), y)$ is also  Cohen-Macaulay. 
            \item   \label{normsurf}
      An isolated surface singularity is normal if and only 
      if it is Cohen-Macaulay. 
           \end{enumerate}
           
           Assume now that $(Y,y)$ satisfies the hypothesis of the proposition.
           
           \begin{itemize}
                  \item If  $(Y,y)$ is Cohen-Macaulay and if the hyperplane 
         section $(f^{-1}(0),y)$  has an isolated singularity, property (\ref{CMsect}) 
         implies that $(f^{-1}(0),y)$ is also Cohen-Macaulay. Property 
         (\ref{normsurf}) implies then that it is normal. 
         
                  \item Conversely, if $(f^{-1}(0),y)$ is normal, then it is Cohen-Macaulay by 
         property (\ref{normsurf}), which implies by definition that $(Y,y)$ is also 
         Cohen-Macaulay.         
            \end{itemize}
   \end{proof}

Let us come back to the smooth projective varieties $B \subset A \subset \bP(V)$. 
The cone $C_B$ is therefore not necessarily normal, even if 
$C_A$ is. But its normalization is easy to describe:

\begin{proposition}  \label{conorm}
    The normalization of $C_B$ is the algebraic variety 
    obtained by contracting the zero-section of the total space 
    of the line bundle 
    $\mathcal{O}(-1)|_{B}$, which is isomorphic to the 
    conormal line bundle of $B$ in $A$.
\end{proposition}

\begin{proof}
   The isomorphism of the two line bundles follows from the fact 
   that $B$ is the vanishing locus of a section of $\mathcal{O}(1)|_{A}$. 
   Here, as is standard in algebraic geometry, $\mathcal{O}(-1)$ 
   denotes the dual of the tautological line bundle 
   on $\bP(V)$. Its fiber above a point of $\bP(V)$ is the associated line. 
   
   Denote by $\tilde{C}_B$ the space obtained by contracting the zero-section of 
   $\mathcal{O}(-1)|_{B}$, and by $\tilde{O} \in \tilde{C}_B$ the image of the $0$-section. 
   By the definition of contractions, $\tilde{C}_B$ is normal (see \cite{GPR 94}). 
   As the fiber of $\mathcal{O}(-1)|_{B}$ over a point $b \in B \hookrightarrow \bP(V)$ is the line of 
   $V$ whose projectivisation is $b$, we see that there is a morphism: 
      \[  \nu : \tilde{C}_B  \to C_B \]
   which induces an isomorphism $\tilde{C}_B \setminus \tilde{O} \simeq C_B \setminus O$. 
   As $\tilde{C}_B$ is normal, by Corollary \ref{normext} and Proposition \ref{charnorm} 
   we see that $\nu$ is a normalization morphism. 
\end{proof}

\bigskip
\section{Isolated singularities with simple elliptic normalization}
\label{sec:sell}

In this section  we apply the method of sweeping out the cone with
hyperplane sections in order to show that the total space of the minimal resolution 
of any non-smoothable simple elliptic surface singularity is diffeomorphic to 
the Milnor fiber of some non-normal isolated surface singularity with simple 
elliptic normalization. We recall first several known properties of ruled surfaces over 
elliptic curves, following Hartshorne's presentation done in \cite[Chapter V.2]{H 77}.  
We conclude with a generalization valid for any 
normal surface singularity, using results of Laufer and Bogomolov \& de Oliveira. 
\medskip

In order to apply the method of the previous section to singularities with 
simple elliptic normalization, we want to find surfaces embedded in some 
projective space which admit a transversal hyperplane section which is 
an elliptic curve. Moreover, because of Propositions \ref{conorm} and 
\ref{sellsmooth}, we would like to get an elliptic curve whose self-intersection 
number in the surface is $\geq 10$. As a consequence of the following theorem of Hartshorne 
\cite{H 69}, this forces us to take a \emph{ruled surface}:

\begin{theorem}
   Let $C$ be a smooth compact curve of genus $g$ on a smooth 
   compact complex algebraic surface $S$. If $S \setminus C$ 
   is minimal (that is, it does not contain smooth rational curves of 
   self-intersection $(-1)$) and $C^2 \geq 4g +6$, then $S$ is 
   a ruled surface and $C$ is a section of the ruling. 
\end{theorem}

Ruled surfaces are those swept by lines (smooth rational curves):

\begin{definition} \label{defruled}
  A {\bf ruled surface} above a smooth projective 
  curve $C$ is a smooth projective surface $X$ together with a 
  surjective morphism $\pi : X \to C$, such that 
  all (scheme-theoretic) fibers are isomorphic to $\bP^1$.
\end{definition}

It is a theorem that all ruled surfaces admit regular sections. 

The following theorem is basic for the classification of ruled 
surfaces (see \cite[Prop. V.2.8, V.2.9]{H 77}):

\begin{theorem} \label{specsect}
    If $\pi: X \to C$ is a ruled surface, it is possible to write 
    $X \simeq \bP(\mathcal{E}^*)$, where $\mathcal{E}$ is a plane bundle 
    on $C$ with the property that $H^0(\mathcal{E}) \neq 0$, but for all 
    line bundles $\mathcal{L}$ on $C$ with $\deg \mathcal{L} < 0$, 
    we have  $H^0(\mathcal{E} \otimes \mathcal{L}) = 0$. In this case the 
    integer $e = - \deg \mathcal{E}$ is an invariant of $X$. Furthermore, in this 
    case there is a section $\sigma_0 : C \to X$ with image $C_0$, such that 
    $\mathcal{O}_X(C_0) \simeq \mathcal{O}_X(1)$. One has $C_0^2 = - e$. 
\end{theorem}

In the sequel, we will say that $e$ is the \emph{numerical invariant} of 
the ruled surface. 

\begin{remark}Ê \label{dualproj}
In fact, Hartshorne writes $\bP(\mathcal{E})$ instead of $\bP(\mathcal{E}^*)$. 
The reason is that his definition of projectivisation is dual to the one 
we use in this paper: instead of taking the lines in a vector space or vector bundle, he takes 
the hyperplanes, that is, the lines in the dual vector space/bundle. 
\end{remark}

We want to find sections of ruled surfaces which appear as hyperplane sections 
for some embedding in a projective space, that is, according to a standard denomination 
of algebraic geometry, \emph{very ample} sections. 
The following proposition combines results contained in  
\cite[Theorems 2.12, 2.15, Exercice 2.12 of Chapter V]{H 77}: 

\begin{proposition} \label{vamp}
    Assume that $C$ is an elliptic curve and that $X$ 
    is a ruled surface above $C$ with numerical invariant $e$. Then: 
      \begin{enumerate}
            \item When $X$ varies for fixed $C$, the invariant $e$ takes all the values 
                in $\Z \cap [-1, \infty)$. 
                
            \item Consider a fixed such ruled surface and let $F$ be one of its fibers. 
                 Take $a \in \Z$. 
                Then the divisor $C_0 + a F$ is very ample on $X$ if and only if $a \geq e + 3$. 
      \end{enumerate}
\end{proposition}

Fix now an integer $a \geq e +3$. By 
Proposition \ref{vamp}, the divisor $C_0 + a F$ is very ample. Denote by 
$X \hookrightarrow \bP(V)$ the associated projective embedding. Let $H$ 
be a hyperplane which intersects it transversally, and let $B := H \cap X$. 
Therefore $B$ is linearly equivalent to $C_0 + a F$ on $X$. We have the 
following intersection numbers on $X$:
\[  \left\{ \begin{array}{l}
    B \cdot F =    (C_0 + a F) \cdot F = C_0 \cdot F = 1 \\
    B^2 = (C_0 + a F)^2 = C_0^2 + 2 a C_0 \cdot F = -e + 2a. 
\end{array} \right. \]
We have used the facts that: 
    \begin{itemize}
        \item  $F$ is a fiber, which implies that $F^2=0$; 
         \item $C_0$ is a section, which implies that $C_0 \cdot F = 1$; 
         \item $C_0^2 = -e$, by Theorem \ref{specsect}. 
     \end{itemize}
The first equality above implies that $B$ is again a section of the ruled surface. 
The second equality shows that a tubular neighborhood of $B$ 
in $X$ is diffeomorphic to a disc bundle over $C$ with Euler number $-e + 2a$. 
As $B$ is a section of the ruling, such a disc bundle may be  chosen as a 
differentiable sub-bundle of the ruling. As the fibers of the ruling $\pi: X \to C$ are spheres, 
its complement is again a disc bundle, 
necessarily of opposite Euler number. Proposition \ref{complsect} shows then that:

\begin{proposition}
    The Milnor fiber of the smoothing $f|_{C_A} : (C_X, 0) \to (\C, 0)$ of the isolated 
    surface singularity $(C_B, 0)$ is diffeomorphic to the disc bundle over $C$ 
    with Euler number $e - 2a$. 
\end{proposition}

\begin{remark} \label{nonorm}
   This shows that the first Betti number of the Milnor fiber of this smoothing 
   is $2$. Greuel and Steenbrink's theorem \ref{vanish} implies 
   that the surface singularity $(C_B, 0)$ which is being smoothed 
   is non-normal. 
\end{remark}

By Proposition \ref{vamp}, we see that the integer $e -2 a$ takes any 
value in $\Z \cap (- \infty, -5]$ (because for fixed $e$, it takes all the integral values in 
$(- \infty, -e-6]$ which have the same parity as $-e-6$). Therefore: 

\begin{itemize}
  \item  this construction applies to  
simple elliptic singularities whose minimal resolution has 
an exceptional divisor with self-intersection any number in $\Z \cap (- \infty, -5]$;

  \item the Milnor fiber is diffeomorphic to the minimal resolution, both being 
     diffeomorphic to the disc bundle over $C$ with Euler number $e - 2a$.   
\end{itemize}

More generally, as an easy consequence of results of Laufer \cite{L 81} 
and Bogomolov and de Oliveira \cite{BO 97}, we have:

\begin{proposition}  \label{milnres}
    Let $(S,0)$ be any normal surface singularity. Then there exists 
    an isolated surface singularity with normalization isomorphic to 
    $(S,0)$, which has a smoothing whose Milnor fibers are diffeomorphic 
    to the minimal resolution of $(S,0)$. 
\end{proposition}

\begin{proof}
  Choose a Milnor representative of $(S,0)$ (see Definition \ref{miln}). 
  Therefore its boundary is strongly pseudo-convex. Take the minimal 
  resolution $\pi : (\Sigma, E)  \to (S,0)$. As $\pi$ is an isomorphism 
  outside $0$, the boundary of $\Sigma$ is also strongly pseudo-convex. 
  By the extensions done in 
  \cite{BO 97} of Laufer's results of \cite{L 81}, there exists a $1$-parameter 
  deformation: 
  $$\psi: (\tilde{\Sigma}, \Sigma) \to (\mathbb{D}_{\epsilon}, 0)$$ 
  of $\Sigma$ over a disc $\mathbb{D}_{\epsilon}$ of radius $\epsilon > 0$, 
  such that the fibers $\Sigma_t$ of $\psi$ above any point 
  $t \in \mathbb{D}_{\epsilon} \setminus 0$ do not 
  contain compact curves. If we choose the disc $\mathbb{D}_{\epsilon}$ small enough, 
  the boundaries of those fibers are also strongly pseudoconvex, 
  by the stability of this property. Therefore, the fibers of $\psi$ above 
  $\mathbb{D}_{\epsilon} \setminus 0$ are all Stein. 
  
  Consider now the \emph{Remmert reduction} (see \cite[Page 229]{GPR 94}): 
    \[ \rho: \tilde{\Sigma} \to \tilde{S}. \]
  By definition, it contracts all the 
  maximal connected compact analytic subspaces of $\tilde{\Sigma}$ to points, 
  and it is normal. 
  The only compact curve of $\tilde{\Sigma}$ is $E$, therefore $\rho$ contracts 
  $E$ to a point $P$, $\tilde{S}$ is a normal $3$-fold and $\rho$ is an 
  isomorphism above $\tilde{S} \setminus P$. As $\tilde{S}$ is normal, Corollary 
  \ref{normext}     shows that the map $\psi$ 
  descends to it, giving us a family:
      $$\psi' : (\tilde{S}, S') \to (\mathbb{D}_{\epsilon}, 0).  $$
  Here $S'$ denotes the fiber of $\psi'$ above the origin. 
  The map $\rho$ being an isomorphism in restriction to $\tilde{\Sigma} \setminus E$, 
  it gives an isomorphism:
    $$ \Sigma \setminus E \simeq S' \setminus P.$$
 Composing it with the isomorphism $\pi^{-1}: S \setminus 0 \to \Sigma \setminus E$, 
 we get an isomorphism $S \setminus 0 \simeq S' \setminus P$ which extends 
 by continuity to $S$. As $S$ is normal, we see that $(S, 0)$ is indeed the normalization of 
 $(S', P)$. 
 
 The map $\psi'$ gives therefore a smoothing with the desired properties:
    \begin{itemize}
         \item Its central fiber $(S', P)$ has normalization isomorphic to $(S,0)$. 
         
         \item Its Milnor fibers are diffeomorphic to the total space of the minimal 
            resolution of $(S,0)$. Indeed, by construction they are isomorphic 
            to the fibers of $\psi$. But $\psi$ is a deformations of a smooth surface, 
            therefore, by Ehresmann's theorem, \emph{all} its fibers are diffeomorphic, 
            and the central fiber is the minimal resolution $\Sigma$ of $S$. 
    \end{itemize}
\end{proof}

Compared with the general result \ref{milnres},  
the advantage of the construction explained before for simple elliptic singularities, 
using the method of sweeping a cone with hyperplane sections, is that it shows  
that in that case the minimal resolution is diffeomorphic to an affine algebraic surface. 

\begin{remark}
   Laufer proved that one can find a $1$-parameter deformation of the total space 
   of the minimal resolution which destroys \emph{any} irreducible component of 
   the exceptional divisor. As for simple elliptic singularities the exceptional divisor 
   is irreducible, we could use his result and proceed as in the previous proof, 
   in order to get the proposition for this special class of singularities.
\end{remark}

\begin{remark}  \label{Kollar}
   After seeing the first version of this paper put on ArXiv, J\'anos Koll\'ar communicated 
   me some information I did not know about papers dealing at least partially with the 
   smoothability of non-normal singularities. The oldest paper he may think about 
   dealing with this problem is \cite[Section 4]{M 78}. There Mumford gives examples 
   of smoothable non-normal isolated surface singularities with simple elliptic normalizations. 
   Extending a result of Mumford's paper, Koll\'ar proved in \cite[Lemma 14.2]{K 95} that 
   all smoothings of isolated surface singularities with rational normalization lift to 
   smoothings of the normalization. This shows that for rational surface singularities, 
   one cannot obtain new Milnor fibers by the method of the present paper. In higher 
   dimensions, Koll\'ar proved in \cite[Theorem 3 (2)]{K 95} that,  if  $X_0$ is a non-normal 
   isolated singularity of dimension at least $3$ whose normalization is log canonical, 
   then $X_0$  is not smoothable: it does not even have normal deformations. He also 
   indicated me \cite[Section 3.1]{Kbook} as a reference for basic material about singularities 
   of cones.
 \end{remark}

\bigskip
\section{Open questions}
\label{sec:opquest}

The following questions are basic for the understanding of the topology 
of the Milnor fibers of isolated, not necessarily normal surface singularities:

\begin{enumerate}
    \item  Given a Milnor fillable contact $3$-manifold, determine whether, 
       up to diffeomorphisms/homeomor\-phisms relative to the boundary,  
       there is always a finite number of Milnor fibers corresponding 
       to smoothings of not-necessarily normal isolated surface 
       singularities filling it.

    \item Given a Milnor fillable contact $3$-manifold $(M, \xi)$, determine whether there exists 
       an isolated surface singularity which fills it, such that its Milnor fibers 
       exhaust, up to diffeomorphisms/homeo\-mor\-phisms, the Milnor fibers of the various isolated 
       singularities which fill $(M, \xi)$.

      \item Given a Milnor fillable contact $3$-manifold $(M, \xi)$, determine whether there exists 
       an isolated surface singularity which fills it, such that its Milnor fibers 
       exhaust, up to diffeomorphisms/homeo\-mor\-phisms, the Stein fillings of $(M, \xi)$.
       
       \item Given a Milnor fillable contact $3$-manifold $(M, \xi)$, classify, 
       up to diffeomorphisms/homeomor\-phisms relative to the boundary, 
       the Milnor fibers of the isolated singularities filling it, and determine 
       the subset of those which appear as Milnor fibers of \emph{normal} 
       singularities.

     \item Determine bounds on the first Betti number of the Milnor fibers  
         of an isolated non-normal surface singularity in terms of its analytic invariants. 
    
\end{enumerate}

\begin{remark}
   For cyclic quotient singularities, Lisca \cite{L 08} proved that there is a finite number 
   of Stein fillings of their contact boundaries and he classified them up to 
   diffeomorphisms relative to the boundary. He conjectured that they are 
   diffeomorphic to the Milnor fibers of the corresponding singularity. 
   N\'emethi and the present author proved this conjecture in \cite{NP 10}. 
   Therefore, in this case the answers of the first three questions are positive and 
   the fourth question is also answered. It would be interesting to understand 
   if the fact that the first three questions have a positive answer is rather an 
   exception or the rule for rational surface singularities. 
\end{remark}

\medskip

\begin{thebibliography}{00} 

  \bibitem{AO 12} Akhmedov, A., Ozbagci, B. \textit{Exotic Stein fillings with arbitrary 
    fundamental groups.} ArXiv:1212.1743v1.
    
    \bibitem{A 66} Artin, M. \textit{On isolated rational singularities of 
         surfaces.}  Amer. J. Math.  \textbf{88}  (1966), 129-136. 

  \bibitem{BO 12} Bhupal, M.,  Ono, K. \textit{Symplectic fillings of links of quotient 
    surface singularities.} Nagoya Math. J. {\bf 207} (2012), 1-45.
    
    \bibitem{BS 13} Bhupal, M., Stipsicz, A. \textit{Smoothings of singularities and symplectic 
       topology.} In \textit{Deformations of surface singularities.} A. N\'emethi and 
       A. Szilard eds., Bolyai  Soc. Math. Studies {\bf 23}, Springer, 2013, 57-97.
    
    \bibitem{BO 97} Bogomolov, F.A., de Oliveira, B. \textit{Stein Small 
    Deformations of Strictly Pseudoconvex Surfaces.} Contemporary 
  Mathematics \textbf{207} (1997), 25-41. 

   \bibitem{CNP 06} Caubel, C., N{\'e}methi, A., Popescu-Pampu, 
  P. \textit{Milnor open books and Milnor fillable contact 
    3-manifolds.}   Topology \textbf{45} (2006), 673-689. 
    
    \bibitem{FS 97} Fintushel, R., Stern, R. J. \textit{Rational blowdowns of smooth 
    $4$-manifolds.} J. Differential Geom. {\bf 46} (1997), no. 2, 181-235.
    
    \bibitem{G 72} Grauert, H. \textit{{\"U}ber 
      die Deformationen isolierter Singularit{\"a}ten analytische Mengen.} 
      Invent. Math. \textbf{15} (1972), 171-198. 
  
   \bibitem{G 78} Greco, S. \textit{Normal varieties.} Notes written with the 
     collaboration of A. Di Sante. Institutiones Mathematicae IV. Academic Press, Inc. 
     [Harcourt Brace Jovanovich, Publishers], London-New York, 1978.
  
  \bibitem{GLS 07} Greuel, G.-M., Lossen, C., Shustin, E. 
     \textit{Introduction to singularities and deformations.}
      Springer, 2007. 
      
  \bibitem{GS 83} Greuel, G.-M., Steenbrink, J. \textit{On the topology
     of smoothable  singularities.} Proc. of Symp. in Pure Maths. \textbf{40} (1983),
  Part 1, 535-545. 
  
   \bibitem{H 69} Hartshorne, R. \textit{Curves with high self-intersection on algebraic 
     surfaces.} Publ. Math. I.H.E.S. {\bf 36} (1969), 111-125. 
  
    \bibitem{H 77} Hartshorne, R. \textit{Algebraic Geometry.} Berlin: Springer,
  1977. 
  
  \bibitem{H 83} Hauser, H. \textit{An algorithm of construction of the semiuniversal deformation 
      of an isolated singularity.}  In \textit{Singularities}, Part 1 (Arcata, Calif., 1981), 567-573, 
      Proc. Sympos. Pure Math. {\bf 40}, Amer. Math. Soc., Providence, R.I., 1983.
      
  \bibitem{H 85} Hauser, H. \textit{La construction de la d\'eformation semi-universelle 
      d'un germe de vari\'et\'e analytique complexe.} Ann. Sci. \'Ecole Norm. Sup. (4) {\bf 18} 
      (1985), no. 1, 1-56.
      
  \bibitem{K 95} Koll\'ar, J. \textit{Flatness criteria.} J. Algebra {\bf 175} (1995), no. 2, 715-727. 
  
  \bibitem{K 13} Koll\'ar, J. \textit{Grothendieck-Lefschetz type theorems for the local 
               Picard group.} J. Ramanujan Math. Soc. {\bf 28A} (2013), 267-285. 
               
   \bibitem{Kbook} Koll\'ar, J. \textit{Singularities of the minimal model program.} 
                With a collaboration of S\'andor Kov\'acs. Cambridge Tracts in Mathematics, 
                {\bf 200}. Cambridge University Press, Cambridge, 2013.
  
  \bibitem{L 72} Laufer, H. \textit{On rational singularities.} Amer. J.
  Math.  {\bf 94} (1972), 597-608.
  
  \bibitem{L 77} Laufer, H.B. \textit{On minimally elliptic
    singularities.} Amer. J. of Math. \textbf{99} (6), 1257-1295 (1977).
    
  \bibitem{L 81} Laufer, H.B. \textit{Lifting cycles to deformations of two-dimensional 
     pseudoconvex manifolds.} Trans. Amer. Math. Soc. {\bf 266} No. 1 (1981), 183-202. 
     
   \bibitem{LT 04} L\^e, D. T., Tosun, M. \textit{Combinatorics of rational singularities.} 
     Comment. Math. Helv. {\bf 79} (2004), no. 3, 582-604.
    
   \bibitem{L 08} Lisca, P. \textit{On symplectic fillings of lens spaces.} 
      Trans. Amer. Math. Soc. {\bf 360} (2008), no. 2, 765-799. 
      
    \bibitem{L 84} Looijenga, E.N. \textit{Isolated singular points on
    complete intersections.} London Math. Soc. Lecture Notes Series
    \textbf{77}, Cambridge Univ. Press (1984).
    
    \bibitem{LW 86} Looijenga, E.N., Wahl, J. \textit{Quadratic functions and smoothing 
      surface singularities.} Topology {\bf 25} (1986), no. 3, 261-291.
    
    \bibitem{M 68} Milnor, J. \textit{Singular Points of Complex
    Hypersurfaces.} Princeton Univ. Press, 1968.
    
    \bibitem{M 78} Mumford, D. \textit{Some footnotes to the work of C. P. Ramanujam.} 
              In  \textit{C. P. Ramanujam--a tribute},  247-262, Tata Inst. Fund. Res. Studies 
               in Math. {\bf 8}, Springer, Berlin-New York, 1978.
    
     \bibitem{N 99} N{\'e}methi, A. \textit{Five lectures on normal surface
    singularities.}
    Lectures delivered at the Summer School in "Low dimensional
    topology", Budapest, Hungary, 1998; Bolyai Society Mathematical
    Studies, 8, Low Dimensional Topology, 269-351, 1999.
    
    \bibitem{N 04} N{\'e}methi, A. \textit{Invariants of normal surface
    singularities.} Contemporary Mathematics \textbf{354}, 161-208
    AMS, 2004. 
    
    \bibitem{N 13} N{\'e}methi, A. \textit{Some meeting points of singularity theory and 
      low dimensional topology.} 
      In \textit{Deformations of surface singularities.} A. N\'emethi and 
       A. Szilard eds., Bolyai  Soc. Math. Studies {\bf 23}, Springer, 2013, 109-162.

    
   \bibitem{NP 10} N{\'e}methi, A., Popescu-Pampu, P.  
        \textit{On the Milnor fibers of cyclic quotient singularities.}
        Proc. London Math. Society (3) {\bf 101} (2010), 554-588; doi: 10.1112/plms/pdq007.
         
    \bibitem{OO 03} Ohta, H., Ono, K. \textit{Symplectic fillings of the 
    link of simple elliptic singularities.} J. Reine 
  Angew. Math. \textbf{565} (2003), 183-205. 
  
    \bibitem{OO 05} Ohta, H., Ono, K. \textit{Simple singularities and 
    symplectic fillings.} J. Differential Geom. \textbf{69} (2005), 
  1-42. 
  
   \bibitem{OO 08} Ohta, H., Ono, K. \textit{Examples of isolated surface singularities whose 
   links have infinitely many symplectic fillings.} J. Fixed Point Theory Appl. {\bf 3} (2008), 
   no. 1, 51-56.
   
   \bibitem{O 14} Ozbagci, B. \textit{On the topology of fillings of contact $3$-manifolds.} 
         To appear in the Proceedings of the Conference ``\emph{Interactions between 
         low dimensional topology and mapping class groups}'' that was held in July 1-5, 2013 
         at the Max Planck Institute for Mathematics, Bonn.
         
    \bibitem{PSS 13} Park, H., Shin, D., Stipsicz, A. I. \textit{Normal complex surface 
        singularities with rational homology disk smoothings.} ArXiv:1311.1929. 
   
   \bibitem{P 74} Pinkham, H. C. \textit{Deformations of algebraic varieties with 
   $\mathbf{G}_m$ action.} Ast\'erisque, No. {\bf 20}. Soci\'et\'e Math\'ematique 
   de France, 1974.
   
  \bibitem{P 07} Popescu-Pampu, P. \textit{The geometry of 
   continued fractions and the topology of surface singularities.} 
  Advanced Stud. in Pure Maths \textbf{46} (2007), 119-195. 
   
   \bibitem{P 08} Popescu-Pampu, P. \textit{On the cohomology rings of holomorphically 
   fillable manifolds.} In \textit{Singularities II. Geometric and
  Topological Aspects.} J. P. Brasselet et al. eds. Contemporary Mathematics 
  \textbf{475}, AMS, 2008, 169-188. 
  
  \bibitem{P 11} Popescu-Pampu, P. \textit{Numerically Gorenstein surface singularities are
  homeomorphic to Gorenstein ones}. Duke Math. Journal {\bf 159} No. 3 (2011), 539-559.
  
  \bibitem{R 97} Reid, M. \textit{Chapters on Algebraic Surfaces.} In
  \textit{Complex Algebraic Geometry.} J. Koll{\'a}r editor, AMS,
  1997, 3-159. 
  
  \bibitem{S 74} Saito, K. \textit{Einfach-elliptische Singularit{\"a}ten.}
  Invent. Math. \textbf{23} (1974), 289-325.
  
  \bibitem{S 68} Schlessinger, M. \textit{Functors of Artin rings.} 
  Trans. Amer. Math. Soc. \textbf{130} (1968), 208-222.
  
  \bibitem{S 71} Schlessinger, M. \textit{Rigidity of quotient
    singularities.} Invent. Math. \textbf{14} (1971), 17-26. 
   
   \bibitem{S 83} Steenbrink, J. \textit{Mixed Hodge structures
    associated with isolated singularities.} Proc. of Symposia in Pure
  Maths. \textbf{40} (1983), Part 2, 513-536. 
  
  \bibitem{S 03} Stevens, J. \textit{Deformations of singularities.} 
  Springer LNM \textbf{1811}, 2003. 
  
  \bibitem{S 09} Stevens, J. \textit{Deforming nonnormal surface singularities.} 
     Manuscript, 2009.
     
 \bibitem{S 13} Stevens, J. \textit{On the classification of rational surface singularities.} 
       J. Singul. {\bf 7} (2013), 108-133.
   
    \bibitem{SSW 08}  Stipsicz, A. I., Szab—, Z., Wahl, J. \textit{Rational blowdowns and 
    smoothings of surface singularities.} J. Topol. {\bf 1} (2008), no. 2, 477-517.
  
    \bibitem{V 80} Varchenko, A.N. \textit{Contact structures and isolated
    singularities.} Mosc. Univ. Math. Bull. \textbf{35} no.2 (1980),
  18-22.
  
    
   \bibitem{W 00} Wall, C.T.C. \textit{Quadratic forms and normal surface
    singularities.} In \textit{Quadratic forms and their
    applications. (Dublin 1999)}, Contemp. Math. \textbf{272},
    293-311, AMS (2000).
    
  \bibitem{GPR 94} \textit{Several complex variables VII.} Grauert, H., Peternell, Th., 
     Remmert, R. eds. \textit{Encyclopaedia of Mathematical Sciences} {\bf 74}, 
     Springer, 1994. 
          
\end{thebibliography}
\end{document}